\documentclass[11pt]{article}
\usepackage{arxiv}
\usepackage[utf8]{inputenc}
\usepackage[T1]{fontenc}
\usepackage{hyperref}
\usepackage{url}
\usepackage{microtype}
\usepackage{amsmath,amssymb,amsfonts,mathtools}
\usepackage{amsthm}
\usepackage{color}
\usepackage{enumitem}
\newtheorem{theorem}{Theorem}
\newtheorem{lemma}{Lemma}
\newtheorem{corollary}[theorem]{Corollary}
\newtheorem{prop}[theorem]{Proposition}
\newtheorem{remark}[theorem]{Remark}

\newcommand{\Sph}{\mathbb{S}^{d-1}}
\newcommand{\weakto}{\Rightarrow}
\newcommand{\ip}[2]{\langle #1,#2\rangle}
\newcommand{\norm}[1]{\lVert #1\rVert}
\newcommand{\push}{\#}

\title{Weak convergence from projections along a positive-measure set of directions}

\author{
	Alejandro Cholaquidis \\
	Centro de Matem\'atica, Facultad de Ciencias \\
	Universidad de la Rep\'ublica \\
	Montevideo, Uruguay \\
	\texttt{acholaquidis@cmat.edu.uy} \\
	\And
	Manuel Hern\'andez-Banadik \\
	Instituto de Estad\'istica, Facultad de Ciencias Econ\'omicas y de Administraci\'on \\
	Universidad de la Rep\'ublica \\
	Montevideo, Uruguay \\
	\texttt{manuel.hernandez@fcea.edu.uy}
}
\date{}

\begin{document}
	
	\maketitle
	
	\begin{abstract}%
		\footnotesize
		The Cram\'er--Wold theorem characterises weak convergence of probability
		measures on \(\mathbb R^d\) through all one-dimensional projected laws.
		We prove a version in which the projections are controlled only on a set
		of directions of positive surface measure and no multivariate limit law
		is prescribed: if the projected laws converge weakly, along every
		direction in such a set, to moment-determinate limits with finite
		moments of all orders, then the sequence converges weakly to the unique
		probability measure whose projections along those directions are the
		given limits. The proof separates tightness, obtained from finitely many
		linearly independent projections, from the identification of
		subsequential limits, obtained through directional moments and
		Petersen's theorem; the identification step is in the spirit of the
		sharp Cram\'er--Wold theorem of Cuesta-Albertos, Fraiman and Ransford.
	\end{abstract}
	
	

	\section{Introduction}
	
	By the classical Cram\'er--Wold theorem~\cite{cramerwold}, weak convergence
	of probability measures on \(\mathbb R^d\) is equivalent to weak
	convergence of all one-dimensional projected laws. This note proves a
	version of this principle in which the projections are controlled only
	along a set of directions of positive surface measure, and no candidate
	limit law on \(\mathbb R^d\) is prescribed. For \(u\in\mathbb S^{d-1}\),
	write \(\pi_u(x):=\langle u,x\rangle\), and denote by \(\pi_u(\mu)\) the
	corresponding projected law of a probability measure \(\mu\) on
	\(\mathbb R^d\). Theorem~\ref{thm:main} states that, if
	\(A\subset\mathbb S^{d-1}\) is a Borel set of positive surface measure and,
	for every \(u\in A\), the laws \(\pi_u(P_n)\) converge weakly to some
	probability law \(N_u\) on \(\mathbb R\) which is moment-determinate with
	finite moments of all orders, then \((P_n)\) converges weakly to the unique
	probability measure \(N\) on \(\mathbb R^d\) whose projections along \(A\)
	are the laws \(N_u\). The limit is thus produced by the directional data:
	its existence and uniqueness are part of the conclusion. The result is a
	qualitative convergence criterion, not a finite-sample reconstruction
	procedure.
	
	The identification underlying the argument is in the spirit of the sharp
	Cram\'er--Wold theorem of Cuesta-Albertos, Fraiman and
	Ransford~\cite{cuesta}: writing \(\pi_x(y):=\langle x,y\rangle\) also for
	\(x\in\mathbb R^d\), their Corollary~3.2 implies, under a Carleman-type
	condition on \(P\), that \(P=Q\) whenever the cone
	\(E(P,Q):=\{x:\pi_x(P)=\pi_x(Q)\}\) has positive Lebesgue measure, which,
	by polar coordinates, amounts to agreement on a set of directions of
	positive surface measure. Proposition~\ref{prop:identification} isolates
	this identification phenomenon, in a formulation whose hypotheses are
	placed directly on the projected laws.
	
	The set \(A\) supplies both ingredients of the proof of
	Theorem~\ref{thm:main}: it contains \(d\) linearly independent directions,
	which yield tightness through the corresponding one-dimensional marginals
	and transfer finite moments from the limit laws \(N_u\) to every
	subsequential limit, so that no moment or determinacy assumption is imposed
	on the latter; its positive surface measure then identifies any two
	subsequential limits through directional moments and Petersen's
	theorem~\cite{petersen}. For related results on multivariate moment
	determinacy, see~\cite{kleiberstoyanov}.
	
	A motivation for this sequential viewpoint comes from projection-based
	inverse problems. If the measures \(P_n\) have densities \(f_n\), the
	density of \(\pi_u(P_n)\) is the one-dimensional projection of \(f_n\)
	along \(u\); under the assumptions of Theorem~\ref{thm:main}, the sequence
	\((P_n)\) then converges weakly to the unique law with projections \(N_u\)
	along \(A\), independently of the reconstruction method used to obtain
	\(P_n\). Such limited-angle settings are common in tomography; see,
	e.g.,~\cite{quinto1993}.

	\section{Main results}
	
	Throughout, \(d\ge2\) is fixed and \(\sigma\) denotes the normalised
	surface measure on \(\Sph\). Recall that \(\pi_u(x):=\ip{u}{x}\) and that
	\(\pi_u(\mu):=\mu\circ\pi_u^{-1}\) is the projected law of a Borel
	probability measure \(\mu\) along \(u\in\Sph\); for a measurable map
	\(T:\mathbb R^d\to\mathbb R^d\), we write \(T_{\push}\mu:=\mu\circ T^{-1}\).
	
	If \(A\subset\Sph\) satisfies \(\sigma(A)>0\), then \(A\) is not contained
	in any proper linear subspace of \(\mathbb R^d\), because every such
	subspace intersects \(\Sph\) in a set of \(\sigma\)-measure zero. Hence
	there exist linearly independent vectors \(u_1,\dots,u_d\in A\), and
	\(T(x):=(\ip{u_1}{x},\dots,\ip{u_d}{x})\) defines a linear isomorphism of
	\(\mathbb R^d\): \(T(x)=0\) forces \(x\perp u_j\) for every \(j\), hence
	\(x=0\). In the proofs below, \(u_1,\dots,u_d\) and \(T\) are fixed in this
	way.
	
	We shall use the following classical consequence of Petersen's theorem
	\cite{petersen}.
	
	\medskip
	\noindent\textbf{Petersen's theorem.}
	Let \(\nu\) be a Borel probability measure on \(\mathbb R^d\) with finite
	moments of all orders. If each coordinate marginal of \(\nu\) is
	moment-determinate, then \(\nu\) is determined by its mixed moments: any
	Borel probability measure on \(\mathbb R^d\) with the same mixed moments
	coincides with \(\nu\).
	
	\begin{lemma} \label{lem:homogeneous-zero-set}
		If \(p\) is a nonzero homogeneous polynomial on \(\mathbb R^d\), then
		$\sigma\{\theta\in\Sph:p(\theta)=0\}=0.$
	\end{lemma}
	
	\begin{proof}
		Let \(F:=\{x\in\mathbb R^d:p(x)=0\}\) and let \(\lambda_d\) denote
		Lebesgue measure on \(\mathbb R^d\). Since \(p\) is a nonzero
		polynomial, it is a real-analytic function on \(\mathbb R^d\) that does
		not vanish identically, so Proposition~1 in \cite{mityagin} gives
		\(\lambda_d(F)=0\). Since \(p\) is homogeneous, say of degree \(m\),
		the set \(F\) is a cone: \(p(r\theta)=r^{m}p(\theta)\) for \(r>0\), so
		that \(\mathbf 1_F(r\theta)=\mathbf 1_{F\cap\Sph}(\theta)\) for all
		\(r>0\) and \(\theta\in\Sph\). Denoting by \(\tilde\sigma\) the
		unnormalised surface measure on \(\Sph\) and by \(B(0,1)\) the unit
		ball of \(\mathbb R^d\), integration in polar coordinates (see, e.g.,
		\cite[Theorem~2.49]{folland}) gives
		\begin{multline*}
			0=\lambda_d\bigl(F\cap B(0,1)\bigr)
			=\int_0^1\!\!\int_{\Sph}\mathbf 1_F(r\theta)\,r^{d-1}\,
			\tilde\sigma(d\theta)\,dr
			=\tilde\sigma\bigl(F\cap\Sph\bigr)\int_0^1 r^{d-1}\,dr
			=\frac{\tilde\sigma(F\cap\Sph)}{d}.
		\end{multline*}
		Hence \(\tilde\sigma(F\cap\Sph)=0\), and therefore
		\(\sigma\{\theta\in\Sph:p(\theta)=0\}=0\).
	\end{proof}
	
	The following identification result, in the spirit of Corollary~3.2
	in \cite{cuesta}, is the key step.
	
	\begin{prop}\label{prop:identification}
		Let \(d\ge2\) and let \(A\subset\Sph\) be a Borel set with
		\(\sigma(A)>0\). Let \(Q_1\) and \(Q_2\) be Borel probability measures
		on \(\mathbb R^d\) such that
		\begin{equation}\label{eq:projection-equality-on-A}
			\pi_u(Q_1)=\pi_u(Q_2)\qquad\text{for every }u\in A.
		\end{equation}
		If, for every \(u\in A\), the common projected law \(\pi_u(Q_1)\) has
		finite moments of all orders and is moment-determinate, then
		\(Q_1=Q_2\).
	\end{prop}
	
	\begin{proof}
		Let \(u_1,\dots,u_d\in A\) and \(T\) be as in the observation above.
		For each \(j=1,\dots,d\), the common projected law
		\(\pi_{u_j}(Q_1)=\pi_{u_j}(Q_2)\) has finite moments of all orders by
		assumption, so \(\int_{\mathbb R^d}|\ip{u_j}{x}|^m\,Q_i(dx)<\infty\)
		for every \(m\ge1\) and \(i=1,2\).
		Since \(u_1,\dots,u_d\) are linearly independent, the map
		\(x\mapsto\sum_{j=1}^d|\ip{u_j}{x}|\) is a norm on \(\mathbb R^d\), so
		there exists \(C>0\) such that
		\(\norm{x}\le C\sum_{j=1}^d|\ip{u_j}{x}|\), whence
		\(\norm{x}^m \le C^m d^{m-1}\sum_{j=1}^d|\ip{u_j}{x}|^m\) for every
		\(m\ge1\). Consequently, \(Q_1\) and \(Q_2\) have finite absolute
		moments of all orders; in particular, every mixed moment of \(Q_1\) and
		\(Q_2\) is finite, each monomial of degree \(m\) in \(x\) being bounded
		in absolute value by \(\norm{x}^m\), so that
		\(\int_{\mathbb R^d}\ip{t}{x}^m\,Q_i(dx)\) is finite for every
		\(t\in\mathbb R^d\) and \(m\ge1\).
		
		Fix \(m\ge1\) and define
		\(r_m(t):=\int_{\mathbb R^d}\ip{t}{x}^m\,Q_1(dx)
		-\int_{\mathbb R^d}\ip{t}{x}^m\,Q_2(dx)\), \(t\in\mathbb R^d\).
		Expanding \(\ip{t}{x}^m\) by the multinomial theorem and integrating
		term by term, \(r_m\) is a homogeneous polynomial of degree \(m\) in
		\(t\), whose coefficients are multinomial coefficients times
		differences of the corresponding mixed moments of order \(m\) of
		\(Q_1\) and \(Q_2\). By \eqref{eq:projection-equality-on-A},
		\(r_m(t)=0\) for every \(t\in A\); if \(r_m\) were not identically
		zero, Lemma~\ref{lem:homogeneous-zero-set} would imply that its zero
		set on \(\Sph\) has \(\sigma\)-measure zero, contradicting
		\(\sigma(A)>0\). Hence \(r_m\equiv0\) and, since multinomial
		coefficients are positive, \(Q_1\) and \(Q_2\) have the same mixed
		moments of order \(m\); as \(m\ge1\) was arbitrary, they have the same
		mixed moments of all orders.
		
		Because \(T\) is linear, each mixed moment of \(T_{\push}Q_i\) is a
		linear combination, with coefficients depending only on \(T\), of mixed
		moments of \(Q_i\); thus \(T_{\push}Q_1\) and \(T_{\push}Q_2\) have the
		same mixed moments of all orders. The measure \(T_{\push}Q_1\) has
		finite moments of all orders, because \(Q_1\) does and \(T\) is linear,
		and, for each \(j=1,\dots,d\), its \(j\)-th coordinate marginal is
		\(\pi_{u_j}(Q_1)\), which is moment-determinate by assumption because
		\(u_j\in A\). Petersen's theorem applied to \(T_{\push}Q_1\) gives
		\(T_{\push}Q_2=T_{\push}Q_1\), and, since \(T\) is a linear
		isomorphism, \(Q_1=Q_2\).
	\end{proof}
	
	We can now state the main result; no multivariate target law appears in
	its hypotheses.
	
	\begin{theorem}
		\label{thm:main}
		Let \((P_n)_{n\ge1}\) be Borel probability measures on
		\(\mathbb R^d\), with \(d\ge2\), and let \(A\subset\mathbb S^{d-1}\) be
		a Borel set with \(\sigma(A)>0\). Assume that, for every \(u\in A\),
		the sequence \(\bigl(\pi_u(P_n)\bigr)_{n\ge1}\) converges weakly to
		some probability measure \(N_u\) on \(\mathbb R\), and that each
		\(N_u\) has finite moments of all orders and is moment-determinate.
		Then there exists a Borel probability measure \(N\) on \(\mathbb R^d\)
		such that \(P_n\weakto N\). Moreover, \(\pi_u(N)=N_u\) for every
		\(u\in A\), and \(N\) is the unique Borel probability measure on
		\(\mathbb R^d\) with this property.
	\end{theorem}
	
	\begin{proof}
		Let \(u_1,\dots,u_d\in A\) and \(T\) be as in the observation above.
		For each \(j=1,\dots,d\), the sequence \(\pi_{u_j}(P_n)\) converges
		weakly by assumption and is therefore tight. Since these laws are the
		coordinate marginals of \(T_{\push}P_n\), the sequence
		\((T_{\push}P_n)\) is tight in \(\mathbb R^d\); as \(T^{-1}\) is
		continuous, \((P_n)\) is tight.
		
		By Prokhorov's theorem, there exists a subsequence
		\(P_{n_k}\weakto Q\). For \(u\in A\), the continuous mapping theorem
		gives \(\pi_u(P_{n_k})\weakto\pi_u(Q)\), while
		\(\pi_u(P_{n_k})\weakto N_u\) by assumption, so \(\pi_u(Q)=N_u\).
		Hence every subsequential weak limit of \((P_n)\) has projections
		\(N_u\) along \(A\); since the \(N_u\) have finite moments of all
		orders and are moment-determinate,
		Proposition~\ref{prop:identification} shows that any two subsequential
		weak limits coincide. Denote by \(N\) their common value. Every
		subsequence of \((P_n)\) has a further weakly convergent subsequence,
		which necessarily converges to \(N\); hence \(P_n\weakto N\), and
		\(\pi_u(N)=N_u\) for every \(u\in A\). Finally, any Borel probability
		measure \(M\) with \(\pi_u(M)=N_u\) for every \(u\in A\) equals \(N\),
		by Proposition~\ref{prop:identification}.
	\end{proof}
	
	When a candidate limit is available, Theorem~\ref{thm:main} specialises to
	the following target-specific criterion.
	
	\begin{corollary}\label{cor:target}
		Let \((P_n)_{n\ge1}\) and \(P\) be Borel probability measures on
		\(\mathbb R^d\), with \(d\ge2\), and let \(A\subset\mathbb S^{d-1}\) be
		a Borel set with \(\sigma(A)>0\). If \(\pi_u(P_n)\weakto\pi_u(P)\) for
		every \(u\in A\), and each projected law \(\pi_u(P)\), \(u\in A\), has
		finite moments of all orders and is moment-determinate, then
		\(P_n\weakto P\).
	\end{corollary}
	
	\begin{proof}
		Apply Theorem~\ref{thm:main} with \(N_u:=\pi_u(P)\): the sequence
		converges weakly to a law \(N\) with \(\pi_u(N)=\pi_u(P)\) for every
		\(u\in A\), and the uniqueness assertion of Theorem~\ref{thm:main}
		gives \(N=P\).
	\end{proof}
	
	\begin{remark}
		The moment and determinacy assumptions in
		Proposition~\ref{prop:identification} and Theorem~\ref{thm:main} are used
		only along the finitely many linearly independent directions
		\(u_1,\dots,u_d\in A\) defining (T). On the rest of \(A\), the argument uses
		only equality of the projected laws, which forces \(r_m\equiv0\) through
		Lemma~\ref{lem:homogeneous-zero-set}.
	\end{remark}
	
	\begin{remark}
		Carleman's condition,
		\(\sum_{m\ge 1}\left(\int_{\mathbb R}|t|^{2m}\,\nu(dt)\right)^{-1/(2m)}=\infty,
		\)
		is sufficient for moment determinacy on \(\mathbb R\); see \cite[Chapter~4]{shohat}
		and \cite{cuesta}. In particular, Gaussian laws on \(\mathbb R\), including
		degenerate ones, have finite moments of all orders and are moment-determinate.
		If all \(N_u\), \(u\in A\), in Theorem~\ref{thm:main} are Gaussian, then the limit
		\(N\) is Gaussian: let \(G\) be the Gaussian law on \(\mathbb R^{d}\) with the
		same mean and covariance as \(N\); for each \(u\in A\), \(\pi_u(G)\) and
		\(\pi_u(N)=N_u\) are one-dimensional Gaussian laws with the same mean and
		variance, hence coincide, and uniqueness gives \(N=G\).
	\end{remark}

	\section*{Acknowledgements}
	We are grateful to the editor, and an anonymous referee for suggestions.

\end{document}